\documentclass{amsart}
\numberwithin{equation}{section}

\newtheorem{thm}{Theorem}[section]
\newtheorem{lem}[thm]{Lemma}
\newtheorem{cor}[thm]{Corollary}

\newtheorem{defin}[thm]{Definition}

\begin{document}

\begin{center}
\textbf{{\large {\ Inverse problem for the subdiffusion equation with non-local in time condition }}}\\[0pt]
\medskip \textbf{Ravshan Ashurov$^{1}$ and Marjona Shakarova$^{2}$}\\[0pt]
\textit{ashurovr@gmail.com\\[0pt]} \textit{shakarova2104@gmail.com\\[0pt]}
\medskip \textit{\ $^{1,2}$ Institute of Mathematics, Academy of Science of Uzbekistan}

\medskip \textit{$^{1}$ School of Engineering, Central Asian University, 264, Milliy Bog Str., Tashkent 111221, Uzbekistan}

\end{center}

\textbf{Abstract}: In the Hilbert space $H$, the inverse problem of determining the right-hand side of the abstract subdiffusion equation with the fractional Caputo derivative is considered. For the forward problem, a non-local in time condition $u(0)=u(T)$ is taken. The right-hand side of the equation has the form $fg(t)$, and the unknown element is $f\in H$. If function $g(t)$ does not change sign, then under a over-determination condition $ u (t_0)= \psi $, $t_0\in (0, T)$, it is proved that the solution of the inverse problem exists and is unique. An example is given showing
the violation of the uniqueness of the solution for some sign-changing functions
$g(t)$. For such functions $g(t)$, under certain conditions on this function, one can achieve well-posedness of the problem by choosing $t_0$. And for some $g(t)$, for the existence of a solution to the inverse problem, certain orthogonality conditions must be satisfied and in this case there is no uniqueness. All the results obtained are also new for the classical diffusion equations.


\section{Introduction}

Let $H$ be a separable Hilbert space  with the scalar product $(\cdot, \cdot)$, and the norm $||\cdot||$ and  $A$ be an self-adjoint positive operator acting in $H$, with a domain of definition  $D(A)$. Assume that the  inverse of $A$ is a compact operator.
Then $A$ has a complete system of orthonormal eigenfunctions ${v_k}$ and a countable set of positive eigenvalues $\lambda_k$: $0<\lambda_1\leq\lambda_2\leq \cdots\rightarrow +\infty$.

Let $C((a,b);H)$ stand for a set of continuous vector-valued functions  (or just functions) $h(t)$,  defined on $t\in (a,b)$. For $h: \mathbb{R}_+\rightarrow H$ fractional analogs of integrals and derivatives are defined in exactly the same way with scalar functions (see, for example, \cite{Liz}). Recall that the Caputo fractional derivative of order $\rho > 0$ of the function $h(t)$ has the form (see, e.g., \cite{Pskhu})
\begin{equation}\label{def0}
D_t^\rho h(t)=\frac{1}{\Gamma
(1-\rho)}\int\limits_0^t\frac{h'(\xi)}{(t-\xi)^{\rho}} d\xi,
\quad t>0,
\end{equation}
provided the right-hand side exists. As usual, $\Gamma(\sigma)$ is
Euler's gamma function.
It is impotent to note that if $\rho=1$, then the fractional derivative coincides with
the ordinary classical derivative: $D_t h(t)= \frac{d}{dt} h(t)$.

Let $\rho\in(0,1] $ be a fixed number.
Consider the following non-local boundary value problem
\begin{equation}\label{prob1}
\left\{
\begin{aligned}
& D_t^\rho u(t)+Au(t)= fg(t), \quad t\in (0,T],\\
&u(T)=u(0),
\end{aligned}
\right.
\end{equation}
where $g(t)\in C[0,T]$ is a given scalar function. $f \in H$ is a known element of $H$.

Problem (\ref{prob1}) is also called \emph{the forward problem}. In fact this problem is well studied in \cite{AshFa}.
The main purpose of this study is the inverse problem of determining the right-hand side of the equation, namely element $f\in H$. To solve the inverse problem one needs an extra condition. Following A.I. Prilepko and A.B. Kostin \cite{Pr} (see also K.B. Sabitov \cite{Sab}) we consider the additional condition in the form:
\begin{equation}\label{ad}
u (t_0) = \psi,\\
\end{equation}
where $t_0$ is a given fixed point of the segment $(0, T)$.

Usually, authors specify an additional condition (\ref{ad}) at the final time $t_0=T$ (see, for example, \cite{Orl}, \cite{Tix} for classical diffusion equations and \cite{MS}, \cite {MS1} for subdiffusion equations). This choice is convenient for real models, since when the process is over it is easy to measure $u(T)$. However, in some cases, choosing $t_0=T$ violates the uniqueness of the solution of the inverse problem, and if $t_0<T$ is chosen, then it is possible to achieve uniqueness in these cases too.

Let us call the non-local boundary value problem (\ref{prob1}) together with the additional condition (\ref{ad}) \emph{the inverse problem} of finding the part $f$ of the right-hand side of the equation.

We will be interested in such solutions that themselves and all the derivatives involved in the equation are continuous in $t$, moreover, all the given functions are continuous and the equation is performed at each point $t$. As an example, let us give the definition of the solution to the inverse problem.
 \begin{defin}\label{def} A pair of functions 
 $\{u(t), f\}$ with the properties $D_t^\rho u(t), Au(t)\in C((0,T]; H)$, $u(t)\in C([0,T];H)$, $f\in H$ satisfying conditions (\ref{prob1})-(\ref{ad}) is called \textbf{the solution} of the inverse problem.
\end{defin}

Inverse problems of determining the right-hand side of various subdiffusion equations have been studied by a number of authors due to the importance of such problems for applications. Note that the methods for studying inverse problems depend on whether $f$ or $g(t)$ is unknown. In the case when the function $g(t)$ is unknown, the inverse problem is equivalently reduced to solving an integral equation (see, for example,\cite{Hand1}, \cite{Ash1}, \cite{Ash2} and \cite{Yama11}).

For the subdiffusion and diffusion equations, the case $g(t)\equiv 1$ and the unknown $f$ has been studied by many authors; in this case, the Fourier method is used to solve the inverse problem (see, for example, \cite{Fur}-\cite{4} and the literature cited in these works).

We note one more paper \cite{AshF} where the authors considered the inverse problem of simultaneous determination of the order of the Riemann-Liouville fractional derivative and the source function in the subdiffusion equations. Using the Fourier method, the authors proved that the solution to this inverse problem exists and is unique.

Let us dwell in more detail on the case $g(t)\not\equiv 1$ and $f$ is unknown, since the present paper is devoted to this case. 

For classical diffusion equations (i.e., $\rho =1$), such an inverse problem has been studied in sufficient detail in the well-known monograph by S. Kabanikhin \cite{Kab1} (Chapter 8) and in the papers \cite{Pr}, \cite{Sab}, \cite{Sab2}, \cite{ Orl}, \cite{Tix}. Since the equation we are considering also includes the diffusion equation, we will dwell on these works in more detail. So, in the works of D.G. Orlovsky \cite{Orl} and I.V. Tikhonov, Yu.S. Eidelman \cite{Tix} abstract diffusion equations in Banach and Hilbert spaces are considered. In \cite{Orl} in the case of a Hilbert space, the elliptic part of the equation is self-adjoint, while in \cite{Tix} both self-adjoint and non-self-adjoint elliptic parts are considered. The uniqueness criterion is found and the existence of a generalized solution is proved. The elliptic part of the diffusion equation in the work of A.I. Prilepko, A.B. Kostin \cite{Pr} is a second-order differential expression. Both non-self-adjoint and self-adjoint elliptic parts are considered. In this article, $g(t)$ also depends on a space variable: $g(t):= g(x,t)$. In the case of a self-adjoint elliptic part, the authors managed to find a criterion for the uniqueness of the generalized solution of the inverse problem. In the works of K.B. Sabitov and A.R. Zaynullov \cite{Sab} and \cite{Sab2}. the elliptic part of the equation is $u_{xx}$, defined on the interval (in \cite{Sab2} and the Laplace operator on the rectangle, respectively. Considering the overdetermination condition in the form (\ref{ad}), a criterion for the uniqueness of the classical solution is found and the Fourier method is used to construct classic solution.

A similar inverse problem for the abstract subdiffusion equation (i.e., $0< \rho < 1$) with the Caputo derivative was studied in \cite{FN}. To define the function $f$, the authors used the following additional condition $\int_0^T u(t)d\mu(t)=u_T$. M. Slodichka et al. \cite{MS} and \cite{MS1} for the subdiffusion equation, the elliptic part of which depends on time, studied the uniqueness of the solution of the inverse problem. It is proved that if the function $g(t)$ does not change sign, then the solution of the inverse problem is unique. It should be especially noted that in \cite{MS1} the authors constructed an example of a function $g(t)$ that changes sign in the domain under consideration, as a result of which the uniqueness of the solution of the inverse problem is lost.

It is well known that the inverse problem under consideration is ill-posed, i.e., the solution does not depend continuously on the given data. Therefore, in papers \cite{S}, \cite{Niu} (see also references there) various regularization methods are proposed for constructing an approximate solution of the inverse problem.

In conclusion, let us dwell on a recent work \cite{Ash3}. In this paper, a similar inverse problem is considered in the case when $A$ is the Laplace operator and the Cauchy condition is taken instead of the non-local condition. In the language of the function $g(t)$, a uniqueness criterion is found, and under the condition that the criterion is satisfied, the existence and uniqueness of a solution to the inverse problem are proved.

The present paper is devoted to the study of the inverse problem (\ref{prob1}), (\ref{ad}) of determining the right-hand side of the equation.

\section{Preliminaries}

In this section, we  recall some information about Mittag-Leffler functions and the solution of forward problem (\ref{prob1}), which we will use in the following sections.

For $0 < \rho < 1$ and an arbitrary complex number $\mu$, let $E_{\rho, \mu}(z)$ denote the Mittag-Leffler function with two parameters of the complex argument $z$:
\begin{equation}\label{ml}
E_{\rho, \mu}(z)= \sum\limits_{k=0}^\infty \frac{z^k}{\Gamma(\rho
k+\mu)}.
\end{equation}
If the parameter $\mu =1$, then we have the classical
Mittag-Leffler function: $ E_{\rho}(z)= E_{\rho, 1}(z)$.

Recall some properties of the Mittag-Leffler functions (see, e.g. \cite{Dzh66}, p. 134 and p. 136).

\begin{lem}\label{mll4} For any $t\geq 0$ one has
\begin{equation*}
|E_{\rho, \mu}(-t)|\leq \frac{C}{1+t},
\end{equation*}
where constant $C$ does not depend on $t$ and $\mu$.
\end{lem}
\begin{lem}\label{MLmonoton} (see \cite{Gor}, p. 47).
The classical Mittag-Leffler function of the negative argument $E_\rho(-t)$ is monotonically
decreasing function for all $0 <\rho < 1$ and
\[
0<E_{\rho} (-t)<1,\quad E_{\rho} (0)=1.
\]
\end{lem}

\begin{lem}\label{ml8} (see \cite{Dzh66}, formula (2.30), p.135 and \cite{AShZun}, lem 4). Let $\mu$ be an arbitrary complex number. Then the following asymptotic estimate holds
\[
\bigg|E_{\rho, \mu}(-t) -\frac{t^{-1}}{\Gamma(\mu-\rho)}\bigg| \leq \frac{C}{t^{2}}, \quad t>1,
\]
where $C$ is an absolute constant.
\end{lem}

\begin{lem}\label{m11}(see \cite{Gor}, formula (4.4.5), p. 61).
Let $\rho > 0 $, $\mu>0$ and $\lambda \in C$. Then for all positive $t$ one has
\begin{equation}\label{MLintFormula}
\int\limits_0^t (t-\eta)^{\mu-1}\eta^{\rho-1}E_{\rho,\rho}(\lambda\eta^\rho)d\eta=t^{\mu+\rho-1} E_{\rho,\rho+\mu}(\lambda t^\rho).
\end{equation}
\end{lem}

To solve the inverse problem we use the following result on the forward problem: If $f\in H$ and $g\in C[0,T]$ then the unique solution of the forward problem has the form (see \cite{AshFa}, Theorem 3 and Corollary 1)
\begin{equation}\label{fp}
       u(t)=\sum\limits_{k=1}^{\infty}\left[\omega_k(t)+\frac{\omega_k(T)}{1-E_{\rho,1}(-\lambda_k T^\rho )} E_{\rho,1}(-\lambda_k t^\rho )  \right ]v_k, 
\end{equation}
    where 
   \[\omega_k(t)=f_k\int\limits_0^{t} (t-\eta)^{\rho-1} E_{\rho, \rho} (-\lambda_k  (t-\eta)^\rho )  g(\eta)d\eta \] and
    $f_k$ is the Fourier coefficients of element $f$.
\

\section{Well-posedness of the inverse problem (\ref{prob1}), (\ref{ad})}

We apply the additional condition (\ref{ad})  to equation (\ref{fp}) and denote by $\psi_k$ the Fourier coefficients of  element $\psi: \psi_k  = (\psi, v_k )$. Then
\begin{equation}\label{EqFor_fk1}
\sum\limits_{k=1}^\infty f_k [(1-E_\rho(-\lambda_k  T))b_{k,\rho}(t_0)+E_\rho(-\lambda_k  t_0)b_{k,\rho}(T)]v_k=
\sum\limits_{k=1}^\infty  \psi_k (1-E_\rho(-\lambda_k  T)) v_k
\end{equation}
where
\[
b_{k,\rho}(t)=\int\limits_0^{t} (t-s)^{\rho-1} E_{\rho, \rho} (-\lambda_k  (t-s)^\rho ) g(s) ds.
\]
Hence, to find $f_k$, we obtain the following equation
\begin{equation}\label{EqFor_fk2}
f_k\Delta_\rho (k, t_0, T) =
  \psi_k (1-E_\rho(-\lambda_k  T)),
\end{equation}
where
$$
\Delta_\rho (k, t_0, T)={(1-E_\rho(-\lambda_k  T))}b_{k,\rho}(t_0)+{E_\rho(-\lambda_k  t_0)b_{k,\rho}(T)}.
$$

  Of course the case $\Delta_\rho (k, t_0, T)=0$ is critical. Since, according to Lemma 2.2
  \[
0<1-E_\rho(-\lambda_k  T)<1, \quad \text{and} \quad 0< E_\rho(-\lambda_k  t_0)<1,
  \]
  then this can happen only when $g(t)$ changes sign. Note that in this case the functions $b_{k,\rho}(t_0)$ and $b_{k,\rho}(T)$ can also change sign.

Let us divide the set of natural numbers $\mathbb{N}$ into two groups $K_{0,\rho}$ and $K_\rho$: $\mathbb{N}=K_\rho \cup K_{0,\rho}$, while the number $k$ is assigned to $K_{0,\rho}$, if $\Delta_\rho (k, t_0, T)=0$, and if $\Delta_\rho (k, t_0, T)\neq 0$, then this number is assigned to $K_\rho$. Note that for some $t_0$ and $T$ the set $K_{0,\rho}$ may be empty, then $K_\rho=\mathbb{N}$. For example, if $g(t)$ is sign-preserving, then $K_\rho=\mathbb{N}$ for all $k$ regardless of the values $t_0$ and $T$.

	Let us establish lower bounds for $\Delta_\rho (k, t_0, T)$.
	First we suppose that $g(t)$ does not change sign. Then $K_{0,\rho}$ is empty.
	
\begin{lem}\label{invvv1} 
Let $\rho \in (0,1]$, $g(t)\in C[0,T]$ and $g(t)\neq 0$, $t\in [0,T]$. Then there are constants $C >0$, depending on $t_0$ and $T$, such that for all $k$:
		\[
		|\Delta_\rho (k, t_0, T)|\geq\frac{C}{\lambda_k}.
		\]
	\end{lem}
\begin{proof} By virtue of the Weierstrass theorem, we have either $g(t)\geq g_+=const >0$ or $g(t)\leq g_-=const < 0$ for all $t\in [0, T]$. It is sufficient to consider the first case; the second is investigated in exactly the same way. For $\tau\in (0, T]$ we have
  \[
  b_{k,\rho}(
  \tau) \geq g_+\int\limits _0^{\tau} \eta^{\rho-1} E_{\rho, \rho} (-\lambda_k  \eta^\rho)d\eta =g_+ \tau ^\rho E_{\rho, \rho+1} (-\lambda_k \tau^\rho ).
  \]
  Apply (see Lemma \ref{ml8})
  \[
  E_{\rho, \rho+1}(-t)=t^{-1} (1- E_{\rho} (-t))
  \]
  to obtain 
  \[
  b_{k,\rho}(\tau)\geq \frac{1}{\lambda_k} (1- E_{\rho} (-\lambda_k \tau^\rho))g_+\geq  \frac{C_{\tau}}{\lambda_k}, \,\, C_{\tau}>0.
  \]  
  Therefore
  $$
\Delta_\rho (k, t_0, T)\geq {(1-E_\rho(-\lambda_k  T))\frac{C_{t_0}}{\lambda_k}}+{E_\rho(-\lambda_k  t_0)\frac{C_{T}}{\lambda_k}}\geq {(1-E_\rho(-\lambda_k  T))\frac{C_{t_0}}{\lambda_k}},
$$ 
and from here, by Lemma \ref{MLmonoton} we obtain the required assertion.
\end{proof}

\begin{thm}\label{thmNotChange}Let $\rho\in (0,1]$, $g(t)\in C[0,T]$ and $g(t)\neq 0$, $t\in [0,T]$. Moreover let $\psi\in D(A)$. Then there exists a unique solution of the inverse problem (\ref{prob1})-(\ref{ad}):
		\begin{equation}\label{f_NotChange}
			f=\sum\limits_{k=1}^\infty f_kv_k,
		\end{equation}
		\begin{equation}\label{u_NotChange}
		u(t)=\sum\limits_{k=1}^{\infty}f_k\left[b_{k,\rho}(t)+\frac{b_{k,\rho}(T)}{1-E_{\rho,1}(-\lambda_k T^\rho )} E_{\rho,1}(-\lambda_k t^\rho )  \right ]v_k,
		\end{equation}
  where
  \[
 f_k= \frac{ \psi_k (1-E_\rho(-\lambda_k  T))}{(1-E_\rho(-\lambda_k  T))b_{k,\rho}(t_0)+E_\rho(-\lambda_k  t_0)b_{k,\rho}(T)}.
  \]
\end{thm}

\begin{proof}

Since according to Lemma \ref{invvv1} $\Delta_\rho (k, t_0, T)\neq 0 $ for all $ k \in \mathbb{N}$, then we get the following equations from (\ref{EqFor_fk2}):
\begin{equation}\label{inv4}
f_k=\frac{ \psi_k (1-E_\rho(-\lambda_k  T))}{\Delta_\rho (k, t_0, T)},
\end{equation}
\begin{equation}\label{inv5}
u_k(t)=f_k\left[b_{k,\rho}(t)+\frac{b_{k,\rho}(T)}{1-E_{\rho,1}(-\lambda_k T^\rho )} E_{\rho,1}(-\lambda_k t^\rho )  \right ].
\end{equation}

With these Fourier coefficients, we have the following series for the unknown functions $f$ and $u(t)$:
\begin{equation}\label{inv10}
f=\sum\limits_{k=1}^\infty \frac{ \psi_k (1-E_\rho(-\lambda_k  T))}{\Delta_\rho (k, t_0, T)}v_k,
\end{equation}
\begin{equation}\label{inv11}
u(t)=\sum\limits_{k=1}^{\infty}f_k\left[b_{k,\rho}(t)+\frac{b_{k,\rho}(T)}{1-E_{\rho,1}(-\lambda_k T^\rho )} E_{\rho,1}(-\lambda_k t^\rho )  \right ]v_k.
\end{equation}
Let $F_j$ be the partial sum of series (\ref{inv10}):
\[F_j=\sum\limits_{k=1}^j \frac{ \psi_k (1-E_\rho(-\lambda_k  T))}{\Delta_\rho (k, t_0, T)}v_k.\]
We show that series $F_{j}$ is absolutely and uniformly convergent.

For this, applying Parseval's equality, we arrive at:
\begin{equation*}
  ||F_{j}||^2=\sum\limits_{k=1}^j \bigg|\frac{ \psi_k (1-E_\rho(-\lambda_k  T))}{\Delta_\rho (k, t_0, T)}\bigg|^2. 
  \end{equation*}
  Since Lemmas \ref{MLmonoton} and \ref{invvv1}, we obtain the following estimate:
 
  \begin{equation*}
      ||F_{j}||^2\leq C\sum\limits_{k=1}^j \lambda_k^2|\psi_k|^2= C ||\psi||^2_1.  
\end{equation*}

Thus, if $\psi \in D(A)$, then from estimates of $F_{j}$ we obtain $f \in H$. 

Let us pass to the consideration of the series \ref{inv11} to determine the function $u(t)$. For this, using the definition of coefficients $f_k$, we study the following series:
\begin{equation}\label{InP}
      u(t)= \sum\limits_{k=1}^\infty\frac{ \psi_k (1-E_\rho(-\lambda_k  T))}{\Delta_\rho (k, t_0, T)} \left[b_{k,\rho}(t)+\frac{b_{k,\rho}(T)}{1-E_{\rho,1}(-\lambda_k T^\rho )} E_{\rho,1}(-\lambda_k t^\rho )  \right ]v_k
     \end{equation}
If $S_j(t)$ is the partial sums of series (\ref{InP}), then by virtue Parseval's equality, we have
\[ ||A S_j(t)||^2= \sum\limits_{k=1}^j \bigg|\frac{ \lambda_k \psi_k (1-E_\rho(-\lambda_k  T))}{\Delta_\rho (k, t_0, T)} \left[b_{k,\rho}(t)+\frac{b_{k,\rho}(T)}{1-E_{\rho,1}(-\lambda_k T^\rho )} E_{\rho,1}(-\lambda_k t^\rho )  \right ]\bigg|^2.\]
According to Lemmas \ref{invvv1} and \ref{MLmonoton} we obtain
\[
||A S_j(t)||^2\leq C \sum\limits_{k=1}^j\lambda_k^2 |\psi_k |^2\left[|b_{k,\rho}(t)|+|b_{k,\rho}(T) |\right]^2
\]
Due to Lemmas \ref{mll4} and \ref{m11} we get 
\[
||A S_j(t)||^2\leq C \sum\limits_{k=1}^j\lambda_k^4 |\psi_k|^2 \max_{0\leq t \leq T }|g(t)|^2[t^\rho E_{\rho,\rho+1}(-\lambda_k t^\rho)+T^\rho E_{\rho,\rho+1}(-\lambda_k T^\rho)]^2
\]
\[
\leq  C \sum\limits_{k=1}^j\lambda_k^4 |\psi_k|^2 \max_{0\leq t \leq T }|g(t)|^2\bigg[\frac{t^\rho }{1+\lambda_kt^\rho} +\frac{T^\rho }{1+\lambda_k T^\rho }\bigg]^2\leq  \sum\limits_{k=1}^j\lambda_k^2 |\psi_k|^2 \max_{0\leq t \leq T }|g(t)|^2.
\]
Hence $ A u(t) \in  C([0,T]; H)$ and in particular $u(t) \in  C([0,T]; H)$.
Further, from equation (\ref{prob1}) one has $D_t^\rho S_j(t)=-A S_j(t)+\sum\limits_{k=1}^j f_k g(t)v_k$, $ t> 0$. Therefore, from
the above reasoning, we have $D_t^\rho u(t) \in  C((0,T]; H)$.

 To prove the uniqueness of the solution, assume the contrary, i.e., there are two different solutions $\{u_1,f_1\}$ and $\{u_2,f_2\}$ satisfying the inverse problem (\ref{prob1} )-(\ref{ad}). We need to show that $u\equiv u_1-u_2 \equiv 0$, $f\equiv f_1-f_2\equiv 0$. For $\{u,f\}$ we have the following problem:
 \begin{equation}\label{prob20}
\left\{
\begin{aligned}
& D_t^\rho u(t)+A u(t) =fg(t),\quad t\in (0,T],\\
&u(T)=u(0), \\
&u(t_0)=0, \quad t_0 \in (0,T).
\end{aligned}
\right.
\end{equation}
We take any solution $\{u,f\}$ and define $u_k=(u,v_k)$ and $f_k=(f,v_k)$. Then, due to the self-adjointness of  operator $A$, we obtain
\[
D_t^\rho u_k(t)= (D_t^\rho u, v_k)= -(A u, v_k)+f_k g(t)=-( u,A v_k)+f_k g(t)=-\lambda_k u_k(t)+f_k g(t).
\]
Therefore, for $u_k$ we have the non-local boundary value problem
\[
D_t^\rho u_k(t)+\lambda_k u_k(t) =f_kg(t),\quad t>0,\quad u_k(T)=u_k(0),
\]
and the additional condition
\[
\quad u_k(t_0)=0.
\]
If $f_k$ is known, then the unique solution of the non-local boundary value problem has the form
\[
u_k(t)= f_k\left[b_{k,\rho}(t)+\frac{b_{k,\rho}(T)}{1-E_{\rho,1}(-\lambda_k T^\rho )} E_{\rho,1}(-\lambda_k t^\rho )\right ].
\]
Apply the additional condition to get
\[
u_k(t_0)= f_k \frac{(1-E_{\rho,1}(-\lambda_k T^\rho ))b_{k,\rho}(t_0)+b_{k,\rho}(T)E_{\rho,1}(-\lambda_k t_0^\rho )}{1-E_{\rho,1}(-\lambda_k T^\rho )} 
\]
\[=\frac{f_k \Delta_\rho (k, t_0, T)}{1-E_{\rho,1}(-\lambda_k T^\rho )}=0.
\]
According to Lemma \ref{invvv1} one has $\Delta_\rho (k, t_0, T)  \neq 0$ and  thanks  to Lemma \ref{MLmonoton} we obtain $1-E_{\rho,1}(-\lambda_k T^\rho ) \neq 0$ for all $k \in \mathbb{N} $. Therefore, $f_k=0$ for all $k$ and due to completeness of the set of eigenfunctions $\{v_k\}$ in $H$, we finally have  $f\equiv 0$ and  $u(t)\equiv0$. \end{proof}

Let us consider the case when $g(t)$ changes sign. In this case, function $\Delta_\rho (k, t_0, T)$ can
become zero, and as a result, the set $K_{0, \rho}$ may turn out to be non-empty. It should also be noted that in this case the solution of the inverse problem may not be unique. Consider the following example.

 Example 1.  Let the inverse problem has the form
 \begin{equation}\label{prob13}
\left\{
\begin{aligned}
& D_t^\rho u(x,t)-u_{xx}(x,t) =f(x)g(t),\quad (x,t)\in (0,\pi) \times(0,T],\\
&u(0,t)=u(\pi,t)=0, \quad t\in (0,T] \\
&u(0)=u(T),  \\
&u (x,\frac{T}{2}) = 0, \quad x\in (0,\pi), \quad t_0 \in (0,T).
\end{aligned}
\right.
\end{equation}
First, note that this problem has a trivial solution $(u, f ) = (0, 0)$.
Now we will show that there is also a non-trivial solution. For this, we take an eigenfunction of the problem $- v_{xx} = \lambda v$
and $v(0)=v(\pi)=0$ with an eigenvalue $\lambda =1$, i.e. $v(x)=\sin x$.  Set $T=1$ and $\omega(t) =(t-\frac{T}{2})^2$. Then, besides the trivial solution, we also have the following non-trivial solution
\[
u(x, t) = \omega(t) v(x),\,\, f(x)=v(x),
\]
of problem (\ref{prob13}) with
\[
g(t)=D_t^\rho \omega(t)+ \omega(t),
\]
i.e.
\[
g(t) =\frac{t^{2-\rho}}{\Gamma(3-\rho)}-\frac{t^{1-\rho}}{\Gamma(2-\rho)}+\left(t-\frac{1}{2}\right)^2.
\]
It can be easily shown that, for example, for the parameter $\rho=0.5$, function $g(t)$ changes its sign.
Indeed, one has
\[
g(0) =\frac{1}{4}> 0\]
and
\[
g(1)=\frac{1}{\Gamma(2.5)}-\frac{1}{\Gamma(1.5)}+\frac{1}{4} = \frac{16+3\sqrt{\pi}-24}{12\sqrt{\pi}}< 0.
\]

Now we should study separately the case of diffusion ($\rho=1$) and subdiffusion ($0<\rho<1$) equations.
	\begin{lem}\label{lemmaClassic}
	Let $\rho=1$, $g(t)\in C^1[0, T]$ and $g(T)\neq 0$. Let $t_0\in (0, T)$ such that $g(t_0) g(T) >0$. Then there exists a number $k_0$ such that, starting from the number $k\geq k_0$, the following estimate holds:
		\begin{equation}\label{estimateClassic}
		|\Delta_1 (k, t_0, T)|\geq\frac{C}{\lambda_k}.
		\end{equation}
 where constant $C$ depends on $k_0$, $t_0$ and $T$.
	\end{lem}
 \begin{proof}For any $\tau\in (0, T]$ by integrating by parts and  the mean value theorem, we get
    \[
    b_{k,1}(\tau)=\int\limits_0^{\tau} e^{-\lambda_k  s} g(\tau-s)ds  =-\frac{1}{\lambda_k}g(\tau-s)  e^{-\lambda_k  s}\bigg|^{\tau}_0 - \frac{1}{\lambda_k}\int\limits_0^{\tau} e^{-\lambda_k  s} g'(\tau-s)ds =
    \]
    \[
    =\frac{1}{\lambda_k} \big[g(\tau)-g(0)e^{-\lambda_k \tau}\big] +  \frac{g'(\xi_k)}{\lambda^2_k}\big[e^{-\lambda_k \tau}
-1\big], \quad \xi_k\in [0, \tau].    
\]
Therefore, there exists a number $k_0$ such that for all $k\geq k_0$ one has
\[
    |b_{k,1}(\tau)|\geq \frac{|g(\tau)|}{2 \lambda_k}.
    \]

    Now, without loss of generality, we assume that $g(T)>0$ and choose $t_0$ so that $g(t_0)>0$. Then
    $$
\Delta_1 (k, t_0, T)\geq {(1-e^{-\lambda_k  T})\frac{g(t_0)}{2\lambda_k}}+{e^{-\lambda_k  t_0}\frac{g(T)}{2\lambda_k}}\geq {(1-e^{-\lambda_k  T})\frac{g(t_0)}{2\lambda_k}},
$$
and this is the assertion of the lemma.
  \end{proof}

 \begin{cor}\label{K1}If conditions of Lemma \ref{lemmaClassic} are  satisfied, then estimate (\ref{estimateClassic}) holds for all $k\in K_1$.
 \end{cor}

 \begin{cor}\label{K01}If conditions of Lemma \ref{lemmaClassic} are  satisfied, then  set $K_{0,1}$ has a finite number elements.
 \end{cor}

 In case of subdiffusion equation ($\rho\in (0,1)$) we have
	\begin{lem}\label{lemmaSub}Let $\rho\in (0,1)$, $g(t)\in C^1[0, T]$ and $g(0)\neq 0$.
Then there exist numbers $T_0>0$ and $k_0$ such that, for all  $T\leq T_0$ and $k\geq k_0$, the following estimates hold:
		\begin{equation}\label{estimateSub}
		|\Delta_{\rho} (k, t_0, T)|\geq\frac{C}{\lambda_k}.
		\end{equation}
  where constants $C$ depends on $T_0$ and $k_0$.
	\end{lem}
 \begin{proof}
  Let $\rho\in (0,1)$ and $\tau \in (0, T]$. Using equality (\ref{MLintFormula}) we integrate by parts and apply the mean value theorem. Then we have
    \[  
    b_{k,\rho}(\tau)=\int\limits_0^{\tau}g(\tau-s) s^{\rho-1} E_{\rho, \rho} (-\lambda_k  s^\rho )  ds=\int\limits_0^{\tau}g(\tau-s) d\big[ s^{\rho} E_{\rho, \rho+1} (-\lambda_k  s^\rho ) \big] =
    \]
    \[
    =g(\tau-s)  s^{\rho} E_{\rho, \rho+1} (-\lambda_k  s^\rho )\bigg|^{\tau}_0+\int\limits_0^{\tau}g'(t_0-s)  s^{\rho} E_{\rho, \rho+1} (-\lambda_k  s^\rho )ds=
    \]
    \[
    =g(0)\,  \tau^{\rho} \,E_{\rho, \rho+1} (-\lambda_k  \tau^\rho)+ g'(\xi_k) \int\limits_0^{\tau} s^{\rho} E_{\rho, \rho+1} (-\lambda_k  s^\rho )ds, \quad \xi_k\in [0, \tau].
    \]
    For the last integral formula (\ref{MLintFormula}) implies
    \[
    \int\limits_0^{\tau} s^{\rho} E_{\rho, \rho+1} (-\lambda_k  s^\rho )ds=\tau^{\rho+1} E_{\rho, \rho+2}(-\lambda_k \tau^\rho).
    \]
    Apply the asymptotic estimate of the Mittag-Leffler functions (Lemma \ref{ml8}) to get
\[
b_{k,\rho}(\tau)=\frac{g(0)}{\lambda_k} +\frac{g'(\xi_k)}{\lambda_k} \tau + O\bigg(\frac{1}{(\lambda_k t_0^\rho)^2}\bigg).
\]

    Now, without loss of generality, we assume that $g(0)>0$.
Then for sufficiently small $\tau$ and sufficiently large $k$ we obtain the lower estimate
\[
b_{k,\rho}(\tau)\geq \frac{g(0)}{2\lambda_k}.
\]
Therefore
  $$
\Delta_\rho (k, t_0, T)\geq (1-E_\rho(-\lambda_k  T))\frac{g(0)}{2\lambda_k}+E_\rho(-\lambda_k  t_0)\frac{g(0)}{2\lambda_k},
$$ 
and since the classical Mittag-Leffler function of the negative argument is monotonically
decreasing function (see Lemma \ref{MLmonoton}), then from here  we obtain the required assertion.
\end{proof}  
 \begin{cor}\label{Krho}If conditions of Lemma \ref{lemmaSub} are  satisfied, then estimate (\ref{estimateSub}) holds for all $T\leq T_0$ and $k\in K_\rho$.
 \end{cor}

 \begin{cor}\label{Krho1}If conditions of Lemma \ref{lemmaSub} are  satisfied and $T$ is sufficiently small, then  set $K_{0,\rho}$ has a finite number elements.
 \end{cor}

 \begin{thm}Let $g(t)\in C^1[0,T]$, $\psi \in D(A)$. Further, we will assume that for $\rho=1$ the conditions of Lemma \ref{lemmaClassic} are satisfied, and for $\rho\in (0,1)$, the conditions of Lemma \ref{lemmaSub} are satisfied and $T$ is sufficiently small.
		
		1) If set $K_{0,\rho}$ is empty, i.e. $\Delta_\rho (k, t_0, T)\neq 0$, for all $k$, then there exists a unique solution of the inverse problem (\ref{prob1})-(\ref{ad}):
		\begin{equation}\label{K0empty_f}
			f=\sum\limits_{k=1}^\infty f_kv_k,
		\end{equation}
		\begin{equation}\label{K0empty_u}
			u(t)=\sum\limits_{k=1}^{\infty}f_k\left[b_{k,\rho}(t)+\frac{b_{k,\rho}(T)}{1-E_{\rho,1}(-\lambda_k T^\rho )} E_{\rho,1}(-\lambda_k t^\rho )  \right ]v_k,
		\end{equation}
		where 
  \[
  f_k= \frac{ \psi_k (1-E_\rho(-\lambda_k  T))}{(1-E_\rho(-\lambda_k  T))b_{k,\rho}(t_0)+E_\rho(-\lambda_k  t_0)b_{k,\rho}(T)}.
  \]
		2) If set $K_{0,\rho}$ is not empty, then for the existence of a solution to the inverse problem, it is necessary and  sufficient that the following  conditions
		\begin{equation}\label{ortogonal}
			\psi_k=(\psi, v_k)=0,\,\, k\in K_{0,\rho},
		\end{equation}
		be satisfied. In this case, the solution to the problem (\ref{prob1})-(\ref{ad}) exists, but is not unique:
		
		\begin{equation}\label{K0notempty_f}
			f=\sum\limits_{k\in K_\rho} \frac{ \psi_k (1-E_\rho(-\lambda_k  T))}{(1-E_\rho(-\lambda_k  T))b_{k,\rho}(t_0)+E_\rho(-\lambda_k  t_0)b_{k,\rho}(T)}v_k+\sum\limits_{k \in K_{0,\rho}} f_k v_k,
		\end{equation}
		\begin{equation}\label{K0notempty_u}
	u(t)=\sum\limits_{k=1}^\infty f_kv_k,
		\end{equation}
		where $f_k$, $k\in K_{0,\rho}$, are arbitrary real numbers.
  \end{thm}
  \begin{proof}
      The proof of the first part of the theorem is completely analogous to the proof of Theorem \ref{thmNotChange}. As regards the proof of the second part of the theorem, we note the following.

If $k\in K_\rho$, then again from (\ref{EqFor_fk2}) we have (\ref{inv4}) and (\ref{inv5}). Then the proof of equalities (\ref{inv4}) and (\ref{inv5}) and the convergence of the corresponding series are proved in the same way as in Theorem \ref{thmNotChange}.

If $k\in K_{0,\rho}$, i.e. $\Delta_\rho (k, t_0, T)=0$, then the solution of equation (\ref{EqFor_fk2}) with respect to $f_k$ exists if and only if the conditions (\ref{ortogonal}) are satisfied. In this case, the solution of the equation can be arbitrary numbers $f_k$. As shown above (see Corollaries \ref{K01} and \ref{Krho1}), under the conditions of the theorem, the set $K_{0,\rho}$, $\rho\in (0, 1]$, contains a finite number of elements. 
 \end{proof}

 \section{Conclusion}
 The work is devoted to one of the important for applications inverse problem - the problem of determining the source function in subdiffusion equations. The unknown right-hand side is given as $f(x)g(t)$, where the function $f(x)$ is unknown. In contrast to the well-known works, the non-local time condition is taken instead of the Cauchy condition. If $g(t)$ do not change sign, then the existence and uniqueness of the solution of the inverse problem with the over-determination condition $u(x, t_0) = \psi(x)$ is proved. An example is constructed showing the lack of uniqueness of the solution for sign-variables $g(t)$. The meaning of taking $t_0$ instead of $T$ in the over-determination condition is that for some sign-changing functions $g(t)$, the choice of $t_0$ can ensure the uniqueness of the solution. In the absence of uniqueness of the solution, the conditions for the orthogonality of the function $\psi(x)$ to some eigenfunctions are found, which ensure the existence of the solution.

 Note that we consider an abstract subdiffusion equation with a self-adjoint operator $A$ in a Hilbert space $H$. The choice of an abstract operator allows us to consider various known models. Scince the operator $A$ is only required to have a complete orthonormal system of eigenfunctions, then as $A$ one can consider any of the elliptic operators given in the work of Ruzhansky et al. \cite{25}. 
 
\section{Acknowledgement}
 The authors are grateful to Sh. Alimov for discussions of these results. The author acknowledges financial support from the  Ministry of Innovative Development of the Republic of Uzbekistan, Grant No F-FA-2021-424.
%%%%%%%%%% References %%%%%%%%%%%%%%%%%%%%%%%%%%%%%%%%

\end{document}